\documentclass[12pt,reqno]{article}

\usepackage{amsmath,amsthm,amsfonts,amssymb,latexsym,mathrsfs,color,cases,url,enumerate}
\usepackage{tikz,tkz-graph}
\usetikzlibrary{decorations.pathreplacing,calligraphy}
\usepackage[square, comma, sort&compress, numbers]{natbib}

\newtheorem{thm}{Theorem}[section]
\newtheorem{lem}[thm]{Lemma}
\newtheorem{cor}[thm]{Corollary}
\newtheorem{prop}[thm]{Proposition}

\newtheorem{prob}[thm]{Problem}
\theoremstyle{definition}

\newtheorem{exa}[thm]{Example}

\newtheorem*{thm1}{Theorem 1.3}
\newtheorem*{thm2}{Theorem 1.4}

\newcommand{\defn}[1]{{\it #1}}

\newcommand{\tbl}{\textcolor{blue}}
\newcommand{\tre}{\textcolor{red}}

\newcommand{\al}{\alpha}
\newcommand{\bal}{\boldsymbol{\al}}
\newcommand{\be}{\beta}
\newcommand{\bbe}{\boldsymbol{\be}}
\newcommand{\del}{\delta}
\newcommand{\ga}{\gamma}
\newcommand{\bga}{\boldsymbol{\ga}}
\newcommand{\Ga}{\Gamma}

\newcommand{\cD}{\mathcal{D}}
\newcommand{\cB}{\mathcal{B}}

\newcommand{\cP}{\mathcal{P}}
\newcommand{\cQ}{\mathcal{Q}}
\newcommand{\cR}{\mathcal{R}}
\newcommand{\cS}{\mathcal{S}}
\newcommand{\cT}{\mathcal{T}}
\newcommand{\cU}{\mathcal{U}}
\newcommand{\cV}{\mathcal{V}}
\newcommand{\cW}{\mathcal{W}}

\newcommand{\rev}{\overleftarrow}

\newcommand{\submat}[2]{%
	\left[
	\begin{array}{c}
		#1 \\
		#2
	\end{array}
	\right]
}
\newcommand{\psubmat}[1]{%
	\left[
	\begin{array}{c}
		#1 \\
		#1
	\end{array}
	\right]
}
\newcommand{\minor}[2]{%
	\left(
	\begin{array}{c}
		#1 \\
		#2
	\end{array}
	\right)
}

\newcommand{\seqi}{i_0, \ldots, i_r}
\newcommand{\seqj}{j_0, \ldots, j_r}

\numberwithin{equation}{section}

\title{A unified planar network approach to total positivity of combinatorial matrices and real-rootedness of polynomials}
\author{Xi Chen%
\thanks{
	Corresponding author.
    \newline\hspace*{3mm}
    {\it Email address:}\quad
    chenxi@dlut.edu.cn (X. Chen)
    },
Lang Fu, Jiajie Ruan
}
\date{\footnotesize
School of Mathematical Sciences, Dalian University of Technology, Dalian 116024, PR China}

\begin{document}
\maketitle

\begin{abstract}
We present a common sufficient condition for the total positivity of combinatorial triangles and their reversals, as well as the real-rootedness of generating functions of the rows.
The proof technique is to construct a unified planar network that represent the matrix, its reversal, and the Toeplitz matrices of rows, respectively,
when selecting different sets of sources and sinks.
These results can be applied to the exponential Riordan arrays, the iteration matrices and the $n$-recursive matrices.
As consequences, we prove the total positivity and real-rootedness properties associated to many well-known combinatorial numbers,
including the Stirling numbers of both kinds (of type A and type B), the Lah numbers, the idempotent numbers, the Delannoy numbers, and the derangement numbers of type A and type B.
\\[1pt]
\end{abstract}

\section{Introduction}
\hspace*{\parindent}
Following Karlin \cite{Kar68},
a finite or infinite matrix $A=[a_{n,k}]_{n,k\ge0}$ is {\it totally positive}
if all its minors are nonnegative.
Such matrices have been widely studied in both pure and applied mathematics \cite{Bre95,Bre96,Kar68,Pin10}.
In recent five years there has been great interest concerning total positivity of combinatorial matrices \cite{CDD+21,CS24,LLY+22,LPW23,LWW23,MMW22,MW22,PSZ23,Sok23,Zhu21,Zhu24}.
However, there are still many combinatorial matrices that are experimentally totally positive but without proof.
For instance, a longstanding conjecture of Brenti \cite[Conjecture 6.10]{Bre96} asked the total positivity of the Eulerian triangle
$$
[A(n,k)]_{n,k\ge0}=
\left[
	\begin{array}{*{6}c}
		1 \\
		1 &  1 \\
		1 &  4 & 1 \\
		1 & 11 & 11 & 1 \\
		1 & 26 & 66 & 26 & 1 \\
		\vdots &&&&& \ddots
	\end{array}
\right],
$$
where the \defn{Eulerian number} $A(n,k)$ counts the number of permutations of $[n+1]$ with $k$ descents,
and satisfies the recurrence
$$
A(n,k)=(n-k+1)A(n-1,k-1)+(k+1)A(n-1,k)
$$
with initial conditions $A(0,k)=\del_{0k}$.
Although the total positivity of the Eulerian triangle is hard to prove,
the first author {\it et al.} \cite{CDD+21},
by using a planar network approach,
showed the total positivity of the reversal of the Stirling triangle of the second kind
$$
\rev{S}:=[\rev{S}(n,k)]_{n,k\ge0}=
\left[
	\begin{array}{*{6}c}
		1 \\
		1 &  1 \\
		1 &  3 &  1 \\
		1 &  6 &  7 & 1 \\
		1 & 10 & 25 & 15 & 1 \\
		\vdots &&&&& \ddots
	\end{array}
\right],
$$
which enjoys a similar recurrence to the Eulerian triangle, i.e.,
$$
\rev{S}(n,k)=(n-k+1)\rev{S}(n-1,k-1)+\rev{S}(n-1,k)
$$
with initial conditions $\rev{S}(0,k)=\del_{0k}$.
On the other hand, it is known \cite{Bre95,Zhu14,CLW15rec} that the Stirling triangle of the second kind is totally positive.
So it is natural to ask the following question.

\begin{prob}\label{prob-TP}
	Given a lower triangular matrix $A=[a_{n,k}]_{n,k\ge0}$,
	under which conditions are both $A$ and its reversal $\rev{A}:=[a_{n,n-k}]_{n,k\ge0}$ totally positive?
\end{prob}

One particular example is the lower triangular matrix whose columns are constant numbers $a_0,a_1,a_2,\dots$.
The reversal of such a triangle is precisely the Toeplitz matrix of the sequence $\bal=(a_n)_{n\ge0}$:
$$
T_{\infty}(\bal):=[a_{i-j}]_{i,j\ge0}=
\left[
	\begin{array}{*{5}c}
		a_0 \\
		a_1 & a_0 \\
		a_2 & a_1 & a_0 \\
		a_3 & a_2 & a_1 & a_0 \\
		\vdots &&&& \ddots
	\end{array}
\right].
$$
An infinite nonnegative sequence $\bal$ is called a \defn{P\'olya frequency sequence} (or shortly, \defn{PF sequence}),
if its Toeplitz matrix $T_{\infty}(\bal)$ is totally positive.
We identify a finite sequence $a_0,a_1,\dots,a_n$ with the infinite sequence $a_0,a_1,\dots,a_n,0,0,\dots$.
A fundamental characterization for PF sequences is due to Schoenberg and Edrei,
which states that the sequence $\bal=(a_n)_{n\ge0}$ is PF if and only if its generating function
\begin{equation}\label{eq-PF}
	\sum_{n\ge0} a_n x^n=C x^k e^{\ga x} \frac{\prod_{j\ge0} (1+\al_j x)}{\prod_{j\ge0} (1-\be_j x)},
\end{equation}
where $C>0$, $k\in\mathbb{N}$, $\al_j,\be_j,\ga\ge0$
and $\sum_{j\ge0}(\al_j+\be_j)<\infty$.
In particular, Aissen, Schoenberg and Whitney states that
a finite sequence of nonnegative numbers is PF if and only if its generating function is \defn{real-rooted}, that is,
it is a constant polynomial or all of its zeros are real numbers (see \cite[p. 399, 412]{Kar68} for instance).
In this sense, we refer the generating functions of PF sequences as \defn{PF formal power series}.
Hence, it gives rise to the another problem.

\begin{prob}\label{prob-RZ}
	Given a lower triangular matrix $A=[a_{n,k}]_{n,k\ge0}$,
	under which conditions is the row generating function $\sum_{k=0}^n a_{n,k} x^k$ real-rooted for each $n\ge0$,
	or equivalently, the Toeplitz matrices of the rows of $A$ are all totally positive?
\end{prob}

The objective of this paper is to give a common answer to Problems \ref{prob-TP} and \ref{prob-RZ}.
An efficient method of proving total positivity of matrices is to prove the total positivity of their (right) production matrices \cite{DFR05},
which has been frequently used in literature \cite{CLW15rio,CLW15rec,PZ16,LLY+22,CS24,Zhu14,Zhu24}.
However, this method does not work when proving total positivity of matrix reversals.
So we turn to consider the total positivity of left production matrices,
and discover that it implies the total positivity of the matrix itself, its reversal, and the Toeplitz matrix of each row, simultaneously.

Let $A$ be a lower triangular matix with nonzero diagonals,
then $A$ is invertible.
Following Liang {\it et al.} \cite{LPW23}, call
\begin{equation}\label{eq-leftprod}
Q(A):=A\cdot
\left[
	\begin{array}{cc}
		1 & O \\
		O & A^{-1}
	\end{array}
\right]
\end{equation}
the \defn{left production matrix} of $A$.
This implies that $Q(A)$ is also lower triangular with nonzero diagonals.
Moreover, we have
\begin{equation}\label{eq-A=Q}
A=Q(A)
\left[
	\begin{array}{cc}
		I_1 &  \\
			& Q(A)
	\end{array}
\right]
\left[
	\begin{array}{cc}
		I_2 &  \\
			& Q(A)
	\end{array}
\right]
\left[
	\begin{array}{cc}
		I_3 &  \\
			& Q(A)
	\end{array}
\right]
\dots.
\end{equation}
Note that one can obtain the matrix $A$ (either invertible or not) from a given matrix $Q(A)$ by \eqref{eq-A=Q}.
In such cases, we also call $Q(A)$ the left production matrix of $A$.

\begin{thm}\label{thm-main}
	Let $A=[a_{n,k}]_{n,k\ge0}$ be a lower triangular matix,
	and $Q(A)$ be its left production matrix.
	If $Q(A)$ is totally positive, then
	\begin{enumerate}[\rm(1)]
		\item $A$ is totally positive;
		\item $\rev{A}=[a_{n,n-k}]_{n,k\ge0}$ is totally positive;
		\item the row generating function $\sum_{k=0}^n a_{n,k} x^k$ is real-rooted for each $n\ge0$.
	\end{enumerate}	
\end{thm}

Recently, Mao {\it et al.} \cite{MMW22} showed Theorem \ref{thm-main} (1) whenever $A$ is a proper Riordan array,
by using matrix decompositions and induction.
However, sometimes Theorem \ref{thm-main} (2) seems more useful.
In fact, it is possible that the total positivity of a given matrix $A$ is not easy to obtain
whenever neither its left nor right production matrix is totally positive (for instance, the matrix $\rev{S}$),
but we can turn to consider its reversal $\rev{A}$ which may have a totally positive left production matrix.
In this paper, we will prove the three results in Theorem \ref{thm-main} simultaneously,
by constructing a unified planar network that represent $A$, $\rev{A}$, and Toeplitz matrices of rows of $A$, respectively,
when selecting different sets of sources and sinks.
Furthermore, we will identify the Toeplitz matrix of each row of $A$ with a submatrix of a certain matrix defined by means of the left production matrix $Q(A)$.

For a given a matrix $A$, denote by
$$
A\submat{\seqi}{\seqj}
\quad \text{and} \quad
A\minor{\seqi}{\seqj}
$$
the submatrix and the minor of $A$ determined by
the rows indexed $i_0 < \cdots < i_r$ and
the columns indexed $j_0 < \cdots < j_r$, respectively.
Let
$$
A_r:=A\psubmat{0,1,\dots,r}
$$
be the $r$th leading principal submatrix (of order $r+1$) of $A$.

\begin{thm}\label{thm-T}
	Let $A=[a_{n,k}]_{n,k\ge0}$ be a lower triangular matix,
	and $Q(A)$ be its left production matrix.
	Define
	\setlength{\arraycolsep}{3pt}
	$$
	M_{n,r}:=
	\left[
		\begin{array}{ccc}
			I_0 \\
				& Q_n(A) \\
				&     & I_r
		\end{array}
	\right]
	\left[
		\begin{array}{ccc}
			I_1 \\
				& Q_n(A) \\
				&     & I_{r-1}
		\end{array}
	\right]
	\dots
	\left[
		\begin{array}{ccc}
			I_r \\
				& Q_n(A) \\
				&     & I_0
		\end{array}
	\right],
	$$
	where $I_k$ is the identity matrix of order $k$ with the convention that $I_0$ is an empty block.
	Let $\bal_n:=(a_{n,0},a_{n,1},\dots,a_{n,n},0,0,\dots)$ be the $n$th row of $A$.
	Then for each $r\ge0$, we have
	$$
	T_r(\bal_n)^T=M_{n,r}\submat{n,n+1,n+2,\dots,n+r}{0,1,2,\dots,r},
	$$
	where $T_r(\bal_n)$ is the $r$th leading principal submatrix of $T_{\infty}(\bal_n)$,
	and $^T$ stands for transpose.
\end{thm}

This paper is organized as follows.
In section 2, we introduce the Lindstr\"om-Gessel-Viennot lemma
and planar network characterizations of totally positive matrices.
In section 3, we give proofs of Theorem \ref{thm-main} and Theorem \ref{thm-T}
by constructing a unified planar network.
As applications, we present sufficient conditions for total positivity and real-rootedness properties in the exponential Riordan arrays and the iteration matrices (section 4),
and the $n$-recursive matrices (section 5), repectively.
Examples include matrices and polynomials formed by the Stirling numbers of both kinds (of type A and type B), the Lah numbers, the idempotent numbers, the Delannoy numbers, and the derangement numbers of type A and type B.

\section{Planar networks and total positivity}
\hspace*{\parindent}
A useful tool in proving the total positivity of matrices
is the famous Lindstr\"om-Gessel-Viennot (LGV) lemma \cite[Chapter 32]{AZ18}.
In this section, we first give a breif overview of the LGV lemma,
and then present characterizations of the total positivity of matrices from the viewpoint of planar networks, which was originated from Brenti \cite{Bre95}.

Let $\cD=(V,E)$ be an acyclic digraph,
which is possibly infinite but locally finite,
that is, there are finitely many paths between any two vertices.
Each edge $e\in E$ is equipped with a weight $w(e)$.
The weight of a directed path $\cP$, denoted by $w(\cP)$,
is defined as the product of the weights of the edges traversed by $\cP$.
Given $u, v\in V$, define
$$
P_{\cD}(u\to v):=\sum_{\cP:\,u\to v} w(\cP),
$$
where the sum ranges over all directed paths going from $u$ to $v$ in $\cD$.
We adopt the convention that $P_{\cD}(u\to u)=1$ for any $u\in V$.
Let $\cU=\{u_0,u_1,u_2,\ldots\}$ and $\cV=\{v_0,v_1,v_2,\ldots\}$
be the sets of \defn{sources} and \defn{sinks} of $\cD$, respectively.
Now assume further that the digraph $\cD$ is planar and that
the sources and sinks lie on the boundary of $\cD$ in the order
``first $U$ in reverse order, then $V$ in order''.
We refer to this setup $(\cD,\cU,\cV)$ as a \defn{planar network}.
(See Figure \ref{fig-standard} for instance.)

\begin{figure}[htbp]
\centering
\begin{tikzpicture}[scale=1.1,>=latex]
	\node at (0,-.5) {$\leftarrow i$};
	\node[right] at (.5,0) {$\uparrow j$};
	\foreach \y in {0,...,4}
		{
		\node[left] at (-4,\y) {$u_{\y}$};
		\node[right] at (0,\y) {$v_{\y}$};
		\foreach \x in {-4,...,0}
			\filldraw (\x,\y) circle(.05);
		}
	\foreach \y in {0,...,3}
		{
		\pgfmathtruncatemacro{\z}{-\y-1}
		\foreach \x in {-4,...,\z}
			{
			\draw[->,thick] (\x,\y)--(\x+.6,\y);
			\draw[thick] (\x+.5,\y)--(\x+1,\y);
			}
		}
	\foreach \y in {1,...,4}
		{
		\foreach \x in {-\y,...,-1}
			{
			\pgfmathtruncatemacro{\i}{-\x}
			\pgfmathtruncatemacro{\j}{\y+\x}
			\draw[->,thick] (\x,\y)--(\x+.6,\y)node[midway,above]{$x_{\i,\j}$};
			\draw[thick] (\x+.5,\y)--(\x+1,\y);
			\draw[->,thick] (\x,\y)--(\x+.6,\y-.6)node[above]{$y_{\i,\j}$};
			\draw[thick] (\x+.5,\y-.5)--(\x+1,\y-1);
			}
		}
\end{tikzpicture}
\caption{A planar network $(\cD,\cU,\cV)$ with sources $\cU=\{u_0,u_1,u_2,u_3,u_4\}$ and sinks $\cV=\{v_0,v_1,v_2,v_3,v_4\}$ having edge weights $x_{i,j}$, $y_{i,j}$ and 1. The edges having weights 0 have been omitted.}
\label{fig-standard}
\end{figure}
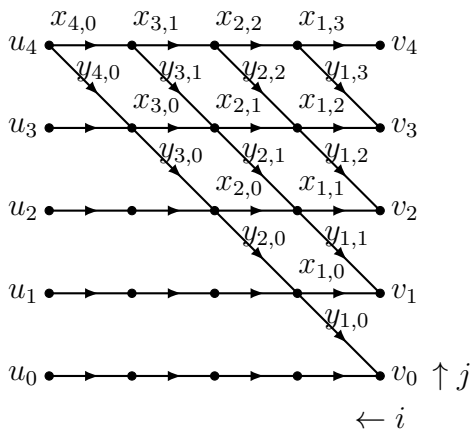

Define the \defn{path matrix} of a planar network $(\cD,\cU,\cV)$ by
\begin{equation}\label{eq-pathmatrix}
P_{\cD}:=\left[P_{\cD}(u_n\to v_k)\right]_{n,k\ge0}.
\end{equation}
We say that two planar networks are \defn{equivalent} if they share the same path matrix.
The LGV lemma shows that every minor of the path matrix $P_{\cD}$
is a sum over families of nonintersecting paths between specific subsets of $\cU$ and $\cV$:
$$
P_{\cD}\minor{n_1,\ldots,n_r}{k_1,\ldots,k_r}=
\sum_{\substack{\sigma\in\mathfrak{S}_r \\
\cP_i:\,u_{n_i}\to v_{k_{\sigma(i)}} \\
\cP_1,\ldots,\cP_r ~\text{nonintersecting}}}
{\rm sgn}(\sigma) w(\cP_1)\cdots w(\cP_r).
$$
Following Brenti~\cite{Bre95}, the sources $\cU=\{u_0,u_1,u_2,\ldots\}$ and sinks $\cV=\{v_0,v_1,v_2,\ldots\}$ are \defn{fully compatible}
if for any subsets $\{u_{n_1},\ldots,u_{n_r}\}\subseteq\cU$ (with $n_1<\cdots<n_r$) and $\{v_{k_1},\ldots,v_{k_r}\}\subseteq\cV$ (with $k_1<\cdots<k_r$),
the only permutation $\sigma$ mapping each source $u_{n_i}$ to the sink $v_{k_{\sigma(i)}}$
which gives rise to a nonempty family of nonintersecting paths in $\cD$
is the identity permutation.
Then for a planar network $(\cD,\cU,\cV)$ whose sources $\cU$ and sinks $\cV$ are fully compatible,
the LGV lemma implies that
$$
P_{\cD}\minor{n_1,\ldots,n_r}{k_1,\ldots,k_r}=
\sum_{\cP_i:\,u_{n_i}\to v_{k_i}}
w(\cP_1)\cdots w(\cP_r).
$$
It follows that if the edge weights of $\cD$ are all nonnegative,
then each minors of $P_{\cD}$ is nonnegative, which means the path matrix $P_{\cD}$ is totally positive.

The next proposition is direct by definition.

\begin{prop}\label{prop-TP-leading}
	An infinite matrix $A$ is totally positive if and only if its $r$th leading principal submatrix $A_r$ is totally positive for each $r\ge0$.
\end{prop}

The following theorem of Brenti \cite{Bre95} was originally stated for finite matrices.
In fact, it can be extended to infinite matrices by applying Proposition \ref{prop-TP-leading}.

\begin{thm}[{\cite[Theorem 3.1]{Bre95}}]\label{thm-Brenti-TP}
	A matrix $A$ is totally positive if and only if
	there exsits a planar network $(\cD,\cU,\cV)$ with sources $\cU$ and sinks $\cV$ fully compatible,
	and edge weights all nonnegative,
	such that $A$ is the path matrix of $\cD$,
	i.e., $A=P_{\cD}$.
\end{thm}

By this characterization, it is possible to present a planar network proof of the next proposition, which can be obtained also from the classic Cauchy-Binet formula.
The technique of constructing planar networks in the following proof will be used repeatedly in the sequel.

\begin{prop}\label{prop-TP-prod}
	If $A$ and $B$ are both totally positive, then so is $AB$.
\end{prop}

\begin{proof}
By Theorem \ref{thm-Brenti-TP}, there exsit planar networks $(\cD^{(1)},\cU^{(1)},\cV^{(1)})$ and $(\cD^{(2)},\cU^{(2)},\cV^{(2)})$ such that their path matrices are $A$ and $B$, respectively.
Suppose that the sources $\cU^{(i)}=\{u_0^{(i)},u_1^{(i)},u_2^{(i)},\dots\}$ and the sinks $\cV^{(i)}=\{v_0^{(i)},v_1^{(i)},v_2^{(i)},\dots\}$.
By identifying $\cV^{(1)}$ and $\cU^{(2)}$ in the sense that $v_k^{(1)}=u_k^{(2)}$ for all $k$,
we then obtain a new planar network $(\cD,\cU^{(1)},\cV^{(2)})$.
Clearly, $AB$ is the path matrix of $\cD$ whose edge weights are all nonnegative.
Note that the fully compatibility of $\cU^{(i)}$ and $\cV^{(i)}$ implies that of $\cU^{(1)}$ and $\cV^{(2)}$.
Hence, again by Theorem \ref{thm-Brenti-TP}, the matrix $AB$ is totally positive.
\end{proof}

Following the notation in \cite{CDD+21},
a planar network $(\cB,\cU,\cV)$ is called \defn{standard binomial-like}
if each vertex of $\cB$ is indexed by a pair $(i,j)$,
with $0\le i\le j$,
where $i$ increases horizontally from right to left
and $j$ increses vertically from bottom to top;
the sources are $\cU=\{u_0,u_1,\dots,u_n\}$
where $u_i=(n,i)$,
and the sinks are $\cV=\{v_0,v_1,\dots,v_n\}$
where $v_j=(0,j)$;
each horizontal directed edge from $(i,j)$ to $(i-1,j)$
has a weight $x_{i,j-i}$ for $i\le j$ and 1 otherwise,
while each diagonal directed edge from $(i,j)$ to $(i-1,j-1)$ has a weight $y_{i,j-i}$ for $i\le j$ and 0 otherwise.
(See Figure \ref{fig-standard} for illustration.)
Clearly, the sources and sinks of a standard binomial-like planar network are fully compatible.
Note that if the edge weights are all 1 in the standard binomial-like planar network,
then the path matrix is precisely the Pascal triangle $\left[\binom{n}{k}\right]_{n,k\ge0}$.
Thus, the Pascal triangle is totally positive by Theorem \ref{thm-Brenti-TP}.

\begin{thm}[{\cite[Theorem 6.10]{Pin10}}]\label{thm-Pinkus}
	A lower triangular matrix $L=[\ell_{i,j}]_{i,j=0}^n$ is totally positive if and only if 
	it can be factored in the form
	\begin{equation}\label{eq-L}
		L=R_1 R_2\cdots R_n,
	\end{equation}
	where $R_k$ is a bidiagonal lower triangular matrix with the $(i+1,i)$-entry equals 0 for $i=1,2,\dots,n-k-1$.
\end{thm}

Combining with Propositions \ref{prop-TP-leading} and \ref{prop-TP-prod}, we are able to restate Theorem \ref{thm-Pinkus} from the viewpoint of planar networks.

\begin{thm}\label{thm-TP}
	A lower triangular matrix $L$ is totally positive if and only if
	there exists a standard binomial-like planar network $(\cB,\cU,\cV)$ whose edge weights are all nonnegative, such that $L$ is the path matrix of $\cB$, i.e., $L=P_{\cB}$.
\end{thm}

\begin{proof}
For a standard binomial-like planar network shown in Figure \ref{fig-standard},
define its ``vertical segments" as shown in Figure \ref{fig-sec}.

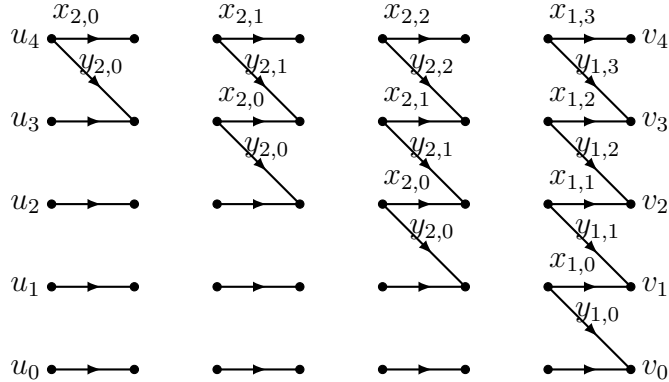
\begin{figure}[htbp]
\centering
\begin{tikzpicture}[scale=1.1,>=latex]
	\foreach \y in {0,...,4}
	{
		\node[left] at (-7,\y) {$u_{\y}$};
		\foreach \x in {-7,-6}
			\filldraw (\x,\y) circle(.05);
	}
	\foreach \y in {0,...,3}
	{
		\draw[->,thick] (-7,\y)--(-6.4,\y);
		\draw[thick] (-6.5,\y)--(-6,\y);
	}
		\draw[->,thick] (-7,4)--(-6.4,4)node[midway,above]{$x_{2,0}$};
		\draw[thick] (-6.5,4)--(-6,4);
		\draw[->,thick] (-7,4)--(-6.4,3.4)node[above]{$y_{2,0}$};
		\draw[thick] (-6.5,3.5)--(-6,3);
	\foreach \y in {0,...,4}
	{
		\foreach \x in {-5,-4}
			\filldraw (\x,\y) circle(.05);
	}
	\foreach \y in {0,1,2}
	{
		\draw[->,thick] (-5,\y)--(-4.4,\y);
		\draw[thick] (-4.5,\y)--(-4,\y);
	}
	\foreach \y in {0,1}
	{
		\draw[->,thick] (-5,\y+3)--(-4.4,\y+3)node[midway,above]{$x_{2,\y}$};
		\draw[thick] (-4.5,\y+3)--(-4,\y+3);
		\draw[->,thick] (-5,\y+3)--(-4.4,\y+2.4)node[above]{$y_{2,\y}$};
		\draw[thick] (-4.5,\y+2.5)--(-4,\y+2);
	}
	\foreach \y in {0,...,4}
	{
		\foreach \x in {-3,-2}
			\filldraw (\x,\y) circle(.05);
	}
	\foreach \y in {0,1}
	{
		\draw[->,thick] (-3,\y)--(-2.4,\y);
		\draw[thick] (-2.5,\y)--(-2,\y);
	}
	\foreach \y in {0,1,2}
	{
		\draw[->,thick] (-3,\y+2)--(-2.4,\y+2)node[midway,above]{$x_{2,\y}$};
		\draw[thick] (-2.5,\y+2)--(-2,\y+2);
		\draw[->,thick] (-3,\y+2)--(-2.4,\y+1.4)node[above]{$y_{2,\y}$};
		\draw[thick] (-2.5,\y+1.5)--(-2,\y+1);
	}
	\foreach \y in {0,...,4}
	{
		\node[right] at (0,\y) {$v_{\y}$};
		\foreach \x in {-1,0}
			\filldraw (\x,\y) circle(.05);
	}
		\draw[->,thick] (-1,0)--(-.4,0);
		\draw[thick] (-.5,0)--(0,0);
	\foreach \y in {0,...,3}
	{
		\draw[->,thick] (-1,\y+1)--(-.4,\y+1)node[midway,above]{$x_{1,\y}$};
		\draw[thick] (-.5,\y+1)--(0,\y+1);
		\draw[->,thick] (-1,\y+1)--(-.4,\y+.4)node[above]{$y_{1,\y}$};
		\draw[thick] (-.5,\y+.5)--(0,\y);
	}
\end{tikzpicture}
\caption{The ``vertical segments" of Figure 1.}
\label{fig-sec}
\end{figure}

Clearly, each bidiagonal matrix $R_k$ in the decomposition \eqref{eq-L} is the path matrix of the $k$th ``vertical segment" of $(\cB,\cU,\cV)$.
Therefore, $L=P_{\cB}$ by Proposition \ref{prop-TP-prod}.
\end{proof}

\section{Proofs of the main results}
\hspace*{\parindent}
In this section we prove first Theorem \ref{thm-main} (1)--(2),
then Theorem \ref{thm-T},
and then Theorem \ref{thm-main} (3).
For convenience, we will restate these theorems.

\begin{thm1}
	Let $A=[a_{n,k}]_{n,k\ge0}$ be a lower triangular matix,
	and $Q(A)$ be its left production matrix.
	If $Q(A)$ is totally positive, then
	\begin{enumerate}[\rm(1)]
		\item $A$ is totally positive;
		\item $\rev{A}=[a_{n,n-k}]_{n,k\ge0}$ is totally positive.
	\end{enumerate}	
\end{thm1}

\begin{proof}
By Propostion \ref{prop-TP-leading}, it suffices to show that the leading principal submatrices $A_m$ and $\rev{A}_m$ are totally positive for each $m\ge0$.
Note that \eqref{eq-A=Q} implies that
\begin{equation}\label{eq-Am=Qm}
A_m=Q_m(A)
\left[
	\begin{array}{cc}
		I_1 &  \\
			& Q_{m-1}(A)
	\end{array}
\right]
\left[
	\begin{array}{cc}
		I_2 &  \\
			& Q_{m-2}(A)
	\end{array}
\right]
\cdots
\left[
	\begin{array}{cc}
		I_m &  \\
			& Q_0(A)
	\end{array}
\right].
\end{equation}
Suppose that $Q(A)$ is totally positive.
Then by Theorem \ref{thm-TP}, there exists a standard binomial-like planar network $(\cB,\cP,\cQ)$ with fully compatible $\cP$ and $\cQ$ and nonnegative edge weights,
such that $Q(A)$ is the path matrix of $\cB$.
Assume that the sources $\cP=\{p_0,p_1,\dots\}$ and the sinks $\cQ=\{q_0,q_1,\dots\}$.
For each $0\le i\le m$, denote by $(\cB_i,\cP_i,\cQ_i)$ the sub-planar network of $(\cB,\cP,\cQ)$ with repect to the sources $\cP_i:=\{p_0,p_1,\dots,p_i\}$ and the sinks $\cQ_i:=\{q_0,q_1,\dots,q_i\}$.
Thus, the path matrix of $\cB_i$ is precisely $Q_i(A)$.

(1)
Combining \eqref{eq-Am=Qm} with the proof of Proposition \ref{prop-TP-prod},
we can obtain a nonegative-weighted planar network for $A_m$,
which has digraph $\cD_m$, sources $\cU_m=\{u_0,u_1,\dots,u_m\}$ and sinks $\cV_m=\{v_0,v_1,\dots,v_m\}$,
where $u_i=(1+\binom{m+1}{2},i)$ and $v_i=(0,i)$,
as shown in Figure \ref{fig-Am}.
The fully compatibility of $\cU_m$ and $\cV_m$ is easy to see.
So $A_m$ is totally positive by Theorem \ref{thm-Brenti-TP}.

\begin{figure}[htbp]
\centering
\begin{tikzpicture}[scale=.6,>=latex]
	\foreach \y in {0,...,4}
		\node[left] at (0,\y) {$u_{\y}$};
	\node[left] at (0,5) {$\vdots$};
	\node[left] at (0,6) {$u_{m}$};
	\foreach \y in {0,...,6}
	{
		\draw (21,\y)--(22,\y);
		\filldraw (22,\y) circle(.08);
	}
	\draw[thick, blue] (0,0) rectangle (6,6);
	\draw[thick, blue] (6,1) rectangle (11,6);
	\draw[thick, blue] (11,2) rectangle (15,6);
	\draw[thick, blue] (15,3) rectangle (18,6);
	\draw[thick, blue] (18,4) rectangle (20,6);
	\draw[thick, blue] (20,5) rectangle (21,6);
	\draw[thick, blue] (21,6) rectangle (22,6);
	\draw (6,0)--(21,0) (11,1)--(21,1) (15,2)--(21,2) (18,3)--(21,3) (20,4)--(21,4);
	\foreach \x in {0,6,11,15,18,20,21}
		\foreach \y in {0,...,6}
			{
				\filldraw (\x,\y) circle(.08);
			}
	\foreach \y in {0,...,4}
		{
			\node[right] at (22,\y) {$v_{\y}$};
		}
	\node[right] at (22,5) {$\vdots$};
	\node[right] at (22,6) {$v_{m}$};
	\node[blue] at (3,3) {$\cB_m$};
	\node[blue] at (8.5,3.5) {$\cdots$};
	\node[blue] at (13,4) {$\cB_4$};
	\node[blue] at (16.5,4.5) {$\cB_3$};
	\node[blue] at (19,5) {$\cB_2$};
	\node[blue] at (20.5,5.5) {$\cB_1$};
	\node[blue] at (21.5,6.5) {$\cB_0$};
\end{tikzpicture}
\caption{A planar network for $A_m$, where $\cB_i$ represents the standard binomial-like planar network for $Q_i(A)$ for $0\le i\le m$. (The edge weights are omitted for the sake of simplicity.)}
\label{fig-Am}
\end{figure}
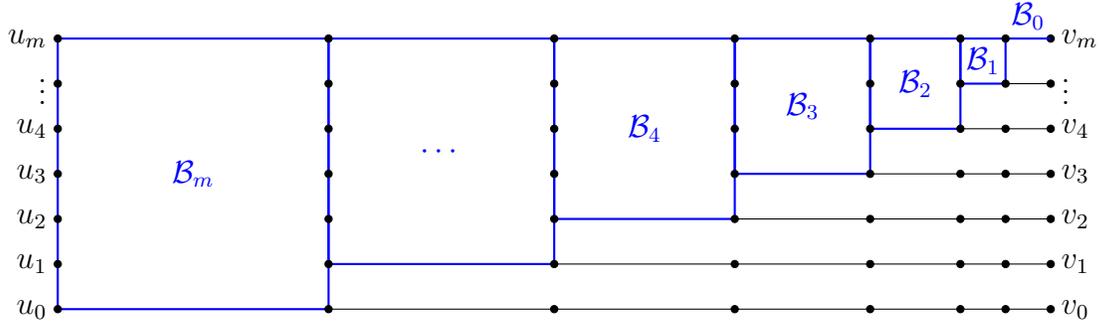

(2)
We claim that there is a planar network for $\rev{A}_m$ that shares the same digraph $\cD_m$ with $A_m$,
but having sources $\cW_m=\{w_0,w_1,\dots,w_m\}$,
and sinks $\rev{\cV}_m:=\{v_0',v_1',\dots,v_m'\}$,
where $w_i=(1+\binom{i+1}{2},m)$ and $v_i'=(0,m-i)$,
as shown in Figure \ref{fig-Am-rev}.

\begin{figure}[htbp]
\centering
\begin{tikzpicture}[scale=.6,>=latex]
	\node[above] at (21,6) {$w_0$};
	\node[above] at (20,6) {$w_1$};
	\node[above] at (18,6) {$w_2$};
	\node[above] at (15,6) {$w_3$};
	\node[above] at (11,6) {$w_4$};
	\node[above] at (6,6) {$\cdots$};
	\node[above] at (0,6) {$w_m$};
	\foreach \y in {0,...,6}
	{
		\draw (21,\y)--(22,\y);
		\filldraw (22,\y) circle(.08);
	}
	\draw[thick, blue] (0,0) rectangle (6,6);
	\draw[thick, blue] (6,1) rectangle (11,6);
	\draw[thick, blue] (11,2) rectangle (15,6);
	\draw[thick, blue] (15,3) rectangle (18,6);
	\draw[thick, blue] (18,4) rectangle (20,6);
	\draw[thick, blue] (20,5) rectangle (21,6);
	\draw[thick, blue] (21,6) rectangle (22,6);
	\draw (6,0)--(21,0) (11,1)--(21,1) (15,2)--(21,2) (18,3)--(21,3) (20,4)--(21,4);
	\foreach \x in {0,6,11,15,18,20,21}
		\foreach \y in {0,...,6}
			{
				\filldraw (\x,\y) circle(.08);
			}
	\foreach \y in {0,...,4}
		{
			\node[right] at (22,6-\y) {$v_{\y}'$};
		}
	\node[right] at (22,1) {$\vdots$};
	\node[right] at (22,0) {$v_{m}'$};
	\node[blue] at (3,3) {$\cB_m$};
	\node[blue] at (8.5,3.5) {$\cdots$};
	\node[blue] at (13,4) {$\cB_4$};
	\node[blue] at (16.5,4.5) {$\cB_3$};
	\node[blue] at (19,5) {$\cB_2$};
	\node[blue] at (20.5,5.5) {$\cB_1$};
	\node[blue] at (21.5,5.5) {$\cB_0$};
\end{tikzpicture}
\caption{A planar network for $\rev{A}_m$, where $\cB_i$ represents the standard binomial-like planar network for $Q_i(A)$ for $0\le i\le m$. (The edge weights are omitted for the sake of simplicity.)}
\label{fig-Am-rev}
\end{figure}

In fact, let $P_m$ be the path matrix of the planar network $(\cD_m,\cW_m,\rev{\cV}_m)$,
that is,
\begin{equation}\label{eq-Pm}
	P_m:=[P_{\cD_m}(w_n\to v_k')]_{0\le n,k\le m}.
\end{equation}
Denote by $A(n;k)$ the $(n,k)$-entry of a given matrix $A$.
Then, for $0\le n,k\le m$., we have
\begin{align}\nonumber
	P_m(n;k)
	=&
	\left[
		\begin{array}{cc}
			I_{m-n} &  \\ \label{eq-PDm}
				& Q_{n}(A)
		\end{array}
	\right]
	\cdots
	\left[
		\begin{array}{cc}
			I_{m-1} &  \\
				& Q_1(A)
		\end{array}
	\right]
	\left[
		\begin{array}{cc}
			I_m &  \\
				& Q_0(A)
		\end{array}
	\right]
	(m;m-k) \\
	=&
	\left[
		\begin{array}{cc}
			I_{m-n} &  \\
				& A_{n}
		\end{array}
	\right]
	(m;m-k),
\end{align}
where the first equation follows from the construction of Figure \ref{fig-Am-rev},
and the second equation follows from \eqref{eq-Am=Qm}.
Hence,
$$
P_m(n;k)=A_n(n;n-k)=A_m(n;n-k)=\rev{A}_m(n;k).
$$
In other words, $\rev{A}_m$ is the path matrix of $(\cD_m,\cW_m,\rev{\cV}_m)$.
It is not hard to prove that $\cW_m$ and $\rev{\cV}_m$ are fully compatible.
Therefore, by Theorem \ref{thm-Brenti-TP}, the matrix $\rev{A}_m$ is totally positive, as desired.
\end{proof}

\begin{thm2}
	Let $A=[a_{n,k}]_{n,k\ge0}$ be a lower triangular matix,
	and $Q(A)$ be its left production matrix.
	Define
	\setlength{\arraycolsep}{3pt}
	\begin{equation}\label{eq-M}
	M_{n,r}:=
	\left[
		\begin{array}{ccc}
			I_0 \\
				& Q_n(A) \\
				&     & I_r
		\end{array}
	\right]
	\left[
		\begin{array}{ccc}
			I_1 \\
				& Q_n(A) \\
				&     & I_{r-1}
		\end{array}
	\right]
	\dots
	\left[
		\begin{array}{ccc}
			I_r \\
				& Q_n(A) \\
				&     & I_0
		\end{array}
	\right],
	\end{equation}
	where $I_k$ is the identity matrix of order $k$ with the convention that $I_0$ is an empty block.
	Let $\bal_n:=(a_{n,0},a_{n,1},\dots,a_{n,n},0,0,\dots)$ be the $n$th row of $A$.
	Then for each $r\ge0$, we have
	\begin{equation}\label{eq-Toep}
		T_r(\bal_n)^T=M_{n,r}\submat{n,n+1,n+2,\dots,n+r}{0,1,2,\dots,r},
	\end{equation}
	where $T_r(\bal_n)$ is the $r$th leading principal submatrix of $T_{\infty}(\bal_n)$,
	and $^T$ stands for transpose.
\end{thm2}

\begin{proof}
In order to show \eqref{eq-Toep},
it suffices to construct a planar network whose path matrix is,
when considering in two different ways,
either $T_r(\bal_n)^T$ or the submatrix of $M_{n,r}$ on the right-hand-side of \eqref{eq-Toep}.
We claim that the desired planar network (as shown in Figure \ref{fig-T} for instance)
shares the same digraph $\cD:=\cD_{n+r}$ with $A_{n+r}$ as constructed in the above proof,
but with sources $\cS_r=\{s_0,s_1,\dots,s_r\}$ and sinks $\cT_r=\{t_0,t_1,\dots,t_r\}$,
where
$$
s_i=\left(1+n+\binom{n+r-i}{2},n+i\right)
\quad\text{and}\quad
t_i=\left(1+\binom{n+r-i}{2},i\right).
$$

\begin{figure}[htbp]
\centering
\begin{tikzpicture}[scale=.55,>=latex]
	\foreach \y in {0,...,6}
		\node[left] at (0,\y) {$u_{\y}$};
	\foreach \y in {0,...,6}
		{
		\foreach \x in {0,...,5}
			\draw (\x,\y)--(\x+1,\y);
		\foreach \x in {0,...,5}
			\filldraw (\x,\y) circle(.08);
		}
	\foreach \y in {1,...,6}
		{
		\pgfmathtruncatemacro{\z}{6-\y}
		\foreach \x in {\z,...,5}
			\draw (\x,\y)--(1+\x,\y-1);
		}
	\foreach \y in {0,...,6}
		{
		\filldraw (6,\y) circle(.08);
		\foreach \x in {6,...,10}
			\draw (\x,\y)--(\x+1,\y);
		\foreach \x in {7,...,10}
			\filldraw (\x,\y) circle(.08);
		}
	\foreach \y in {2,...,6}
		{
		\pgfmathtruncatemacro{\z}{12-\y}
		\foreach \x in {\z,...,10}
			\draw (\x,\y)--(1+\x,\y-1);
		}
	\foreach \y in {0,...,6}
		{
		\filldraw (11,\y) circle(.08);
		\foreach \x in {11,...,14}
			\draw (\x,\y)--(\x+1,\y);
		\foreach \x in {12,13,14}
			\filldraw (\x,\y) circle(.08);
		}
	\foreach \y in {3,...,6}
		{
		\pgfmathtruncatemacro{\z}{17-\y}
		\foreach \x in {\z,...,14}
			\draw (\x,\y)--(1+\x,\y-1);
		}
	\foreach \y in {0,...,6}
		{
		\filldraw (15,\y) circle(.08);
		\foreach \x in {15,16,17}
			\draw (\x,\y)--(\x+1,\y);
		\foreach \x in {16,17}
			\filldraw (\x,\y) circle(.08);
		}
	\foreach \y in {4,5,6}
		{
		\pgfmathtruncatemacro{\z}{21-\y}
		\foreach \x in {\z,...,17}
			\draw (\x,\y)--(1+\x,\y-1);
		}
	\foreach \y in {0,...,6}
		{
		\filldraw (18,\y) circle(.08);
		\foreach \x in {18,19}
			\draw (\x,\y)--(\x+1,\y);
		\foreach \x in {19}
			\filldraw (\x,\y) circle(.08);
		}
	\foreach \y in {5,6}
		{
		\pgfmathtruncatemacro{\z}{24-\y}
		\foreach \x in {\z,...,19}
			\draw (\x,\y)--(1+\x,\y-1);
		}
	\foreach \y in {0,...,6}
		{
		\filldraw (20,\y) circle(.08);
		\foreach \x in {20}
			\draw (\x,\y)--(\x+1,\y);
		\foreach \x in {21}
			\filldraw (\x,\y) circle(.08);
		}
	\foreach \y in {6}
		{
		\pgfmathtruncatemacro{\z}{26-\y}
		\foreach \x in {\z,...,20}
			\draw (\x,\y)--(1+\x,\y-1);
		}
	\foreach \y in {0,...,6}
		{
		\draw (21,\y)--(22,\y);
		\filldraw (22,\y) circle(.08);
		}
	\foreach \y in {0,...,6}
		\node[right] at (22,\y) {$v_{\y}$};
	\filldraw[red] (4,2) circle(.12) node[below]{$s_0$};
	\filldraw[red] (9,3) circle(.12) node[below]{$s_1$};
	\filldraw[red] (13,4) circle(.12) node[below]{$s_2$};
	\filldraw[red] (16,5) circle(.12) node[below]{$s_3$};
	\filldraw[red] (18,6) circle(.12) node[below]{$s_4$};
	\filldraw[blue] (6,0) circle(.12) node[below]{$t_0$};
	\filldraw[blue] (11,1) circle(.12) node[below]{$t_1$};
	\filldraw[blue] (15,2) circle(.12) node[below]{$t_2$};
	\filldraw[blue] (18,3) circle(.12) node[below]{$t_3$};
	\filldraw[blue] (20,4) circle(.12) node[below]{$t_4$};
	\draw[decorate,decoration={brace}] (5.8,-1)--(0.2,-1) node[midway,below]{$G_0$};
	\draw[decorate,decoration={brace}] (10.8,-1)--(6.2,-1) node[midway,below]{$G_1$};
	\draw[decorate,decoration={brace}] (14.8,-1)--(11.2,-1) node[midway,below]{$G_2$};
	\draw[decorate,decoration={brace}] (17.8,-1)--(15.2,-1) node[midway,below]{$G_3$};
	\draw[decorate,decoration={brace}] (19.8,-1)--(18.2,-1) node[midway,below]{$G_4$};
	\draw[decorate,decoration={brace}] (21,-1)--(20.1,-1) node[midway,below]{$G_5$};
	\draw[decorate,decoration={brace}] (22,-1)--(21.1,-1) node[midway,below]{$G_6$};
\end{tikzpicture}
\caption{A planar network $(\cD,\cS_4,\cT_4)$ for $T_4(\bal_2)^T$
(edge weights are ommited for briefness).
The ``vertical segments'' of $(\cD,\cU_{6},\cV_{6})$ are denoted by $G_0,G_1,\dots,G_{6}$.}
\label{fig-T}
\end{figure}
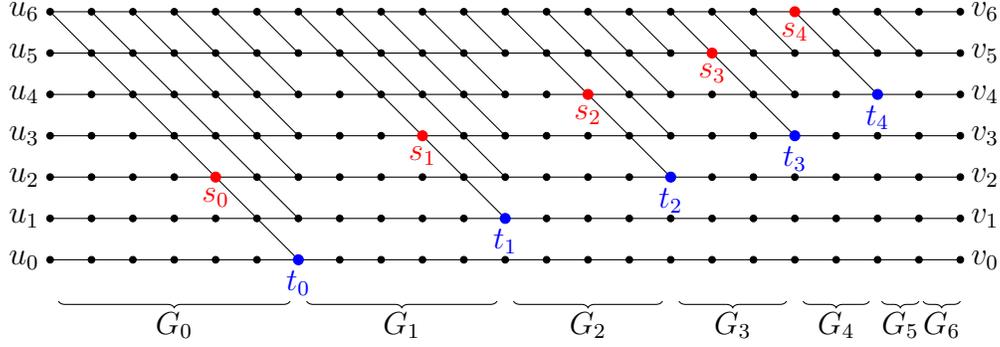

In fact, let 
$
P_r:=[P_{\cD}(s_i\to t_j)]_{0\le i,j\le r}
$
be the path matrix of $(\cD,\cS_r,\cT_r)$,
then the structure of the planar network implies that $P_{\cD}(s_i\to t_j)=0$ whenever $i>j$ or $j-i>r$,
and $P_{\cD}(s_i\to t_j)=P_{\cD}(s_0\to t_{j-i})$ whenever $0\le j-i\le r$.
That is, $P_r$ is the transpose of a Toeplitz matrix.
Moreover, we have
$$
P_{\cD}(s_0\to t_{j-i})=P_{\cD}(u_n\to v_{j-i})=a_{n,j-i}
$$
since the unique path from $u_n$ to $s_0$ has weight 1,
and so does the unique path from $t_{j-i}$ to $v_{j-i}$.
Hence, the path matrix $P_r=T_r(\bal_n)^T$.

On the other hand,
let $\cU'_{n+r}=\{u_n,u_{n+1},\dots,u_{n+r}\}$ and $\cV'_{n+r}=\{v_n,v_{n+1},\dots,v_{n+r}\}$,
then it is easy to see that the planar network $(\cD,\cS_r,\cT_r)$ is equivalent to $(\cD',\cU'_{n+r},\cV'_{n+r})$,
where $\cD'$ is obtained by deleting all the diagonal edges ``out of range'' of the $s_i$'s and $t_i$'s,
and then moving the sources to the leftmost and the sinks to the rightmost by identifying $s_i$ with $u_{n+i}$, and $t_i$ with $v_i$, respectively.
(See Figure \ref{fig-T-2} for instance.)
Note that for each $0\le i\le n+r$,
the ``vertical segment'' $G_i$ of $(\cD,\cU_{n+r},\cV_{n+r})$ represents the matrix
\setlength{\arraycolsep}{2pt}
$
\left[
	\begin{array}{cc}
		I_i \\
		          & Q_{n+r-i}(A)
	\end{array}
\right];
$
furthermore, for each $0\le i\le r$,
the ``vertical segment'' $H_i$ of $(\cD',\cU_{n+r},\cV_{n+r})$ represents the matrix
\setlength{\arraycolsep}{2pt}
$
\left[
	\begin{array}{ccc}
		I_{i} \\
		      & Q_n(A) \\
			  &        & I_{r-i}
	\end{array}
\right].
$
Thus, combining this with \eqref{eq-M}, we derive that the path matrix of $(\cD',\cU_{n+r},\cV_{n+r})$ is precisely $M_{n,r}$.
Hence, the path matrix of $(\cD',\cU'_{n+r},\cV'_{n+r})$ is
$$
[P_{\cD'}(s_i\to t_j)]_{0\le i,j\le r}=M_{n,r}\submat{n,n+1,n+2,\dots,n+r}{0,1,2,\dots,r}
$$
Recall that eqivalent planar networks have the same path matrix.
Therefore,
$$
T_r(\bal_n)^T=M_{n,r}\submat{n,n+1,n+2,\dots,n+r}{0,1,2,\dots,r},
$$
as desired.
\end{proof}

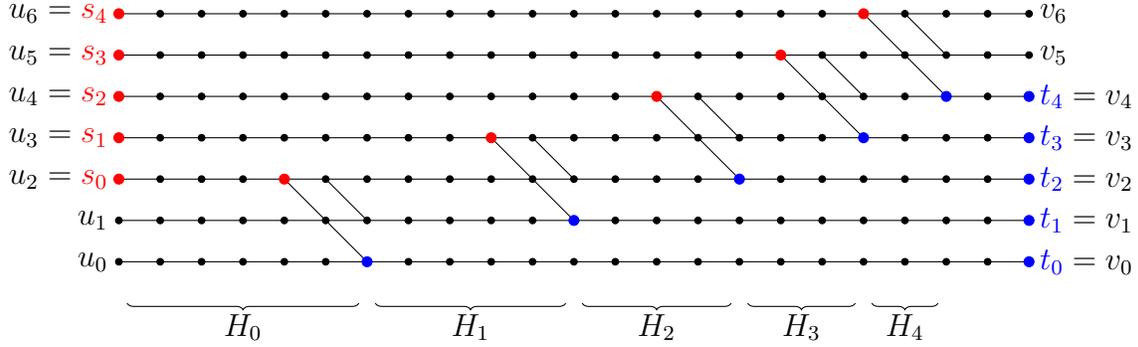
\begin{figure}[htbp]
\centering
\begin{tikzpicture}[scale=.55,>=latex]
	\foreach \y in {0,1}
		\node[left] at (0,\y) {$u_{\y}$};
	\node[left] at (0,2) {$u_2=\tre{s_0}$};
	\node[left] at (0,3) {$u_3=\tre{s_1}$};
	\node[left] at (0,4) {$u_4=\tre{s_2}$};
	\node[left] at (0,5) {$u_5=\tre{s_3}$};
	\node[left] at (0,6) {$u_6=\tre{s_4}$};
	\foreach \y in {0,...,6}
		{
		\foreach \x in {0,...,5}
			\draw (\x,\y)--(\x+1,\y);
		\foreach \x in {0,...,5}
			\filldraw (\x,\y) circle(.08);
		}
	\foreach \y in {1,2}
		{
		\pgfmathtruncatemacro{\z}{6-\y}
		\foreach \x in {\z,...,5}
			\draw (\x,\y)--(1+\x,\y-1);
		}
	\foreach \y in {0,...,6}
		{
		\filldraw (6,\y) circle(.08);
		\foreach \x in {6,...,10}
			\draw (\x,\y)--(\x+1,\y);
		\foreach \x in {7,...,10}
			\filldraw (\x,\y) circle(.08);
		}
	\foreach \y in {2,3}
		{
		\pgfmathtruncatemacro{\z}{12-\y}
		\foreach \x in {\z,...,10}
			\draw (\x,\y)--(1+\x,\y-1);
		}
	\foreach \y in {0,...,6}
		{
		\filldraw (11,\y) circle(.08);
		\foreach \x in {11,...,14}
			\draw (\x,\y)--(\x+1,\y);
		\foreach \x in {12,13,14}
			\filldraw (\x,\y) circle(.08);
		}
	\foreach \y in {3,4}
		{
		\pgfmathtruncatemacro{\z}{17-\y}
		\foreach \x in {\z,...,14}
			\draw (\x,\y)--(1+\x,\y-1);
		}
	\foreach \y in {0,...,6}
		{
		\filldraw (15,\y) circle(.08);
		\foreach \x in {15,16,17}
			\draw (\x,\y)--(\x+1,\y);
		\foreach \x in {16,17}
			\filldraw (\x,\y) circle(.08);
		}
	\foreach \y in {4,5}
		{
		\pgfmathtruncatemacro{\z}{21-\y}
		\foreach \x in {\z,...,17}
			\draw (\x,\y)--(1+\x,\y-1);
		}
	\foreach \y in {0,...,6}
		{
		\filldraw (18,\y) circle(.08);
		\foreach \x in {18,19}
			\draw (\x,\y)--(\x+1,\y);
		\foreach \x in {19}
			\filldraw (\x,\y) circle(.08);
		}
	\foreach \y in {5,6}
		{
		\pgfmathtruncatemacro{\z}{24-\y}
		\foreach \x in {\z,...,19}
			\draw (\x,\y)--(1+\x,\y-1);
		}
	\foreach \y in {0,...,6}
		{
		\filldraw (20,\y) circle(.08);
		\foreach \x in {20}
			\draw (\x,\y)--(\x+1,\y);
		\foreach \x in {21}
			\filldraw (\x,\y) circle(.08);
		}
	\foreach \y in {0,...,6}
		{
		\draw (21,\y)--(22,\y);
		\filldraw (22,\y) circle(.08);
		}
	\foreach \y in {0,...,4}
		\node[right] at (22,\y) {$\tbl{t_{\y}}=v_{\y}$};
	\foreach \y in {5,6}
		\node[right] at (22,\y) {$v_{\y}$};
	\foreach \x in {0,4}
		\filldraw[red] (\x,2) circle(.12);
	\foreach \x in {0,9}
		\filldraw[red] (\x,3) circle(.12);
	\foreach \x in {0,13}
		\filldraw[red] (\x,4) circle(.12);
	\foreach \x in {0,16}
		\filldraw[red] (\x,5) circle(.12);
	\foreach \x in {0,18}
		\filldraw[red] (\x,6) circle(.12);
	\foreach \x in {6,22}
		\filldraw[blue] (\x,0) circle(.12);
	\foreach \x in {11,22}
		\filldraw[blue] (\x,1) circle(.12);
	\foreach \x in {15,22}
		\filldraw[blue] (\x,2) circle(.12);
	\foreach \x in {18,22}
		\filldraw[blue] (\x,3) circle(.12);
	\foreach \x in {20,22}
		\filldraw[blue] (\x,4) circle(.12);
	\draw[decorate,decoration={brace}] (5.8,-1)--(0.2,-1) node[midway,below]{$H_0$};
	\draw[decorate,decoration={brace}] (10.8,-1)--(6.2,-1) node[midway,below]{$H_1$};
	\draw[decorate,decoration={brace}] (14.8,-1)--(11.2,-1) node[midway,below]{$H_2$};
	\draw[decorate,decoration={brace}] (17.8,-1)--(15.2,-1) node[midway,below]{$H_3$};
	\draw[decorate,decoration={brace}] (19.8,-1)--(18.2,-1) node[midway,below]{$H_4$};
\end{tikzpicture}
\caption{The planar network $(\cD',\cU'_6,\cV'_6)$ equivalent to $(\cD,\cS_4,\cT_4)$ in Figure \ref{fig-T}
(edge weights are ommited for briefness).
The ``vertical segments'' of $(\cD',\cU_{6},\cV_{6})$ are denoted by $H_0,H_1,\dots,H_{6}$.}
\label{fig-T-2}
\end{figure}

Now we are ready to prove Theorem \ref{thm-main} (3).

\begin{thm1}
	Let $A=[a_{n,k}]_{n,k\ge0}$ be a lower triangular matix,
	and $Q(A)$ be its left production matrix.
	If $Q(A)$ is totally positive, then
	\begin{enumerate}[\rm(1)]
		\item[\rm(3)] the row generating function $\sum_{k=0}^n a_{n,k} x^k$ is real-rooted for each $n\ge0$.
	\end{enumerate}	
\end{thm1}

\begin{proof}
For each $n\ge0$, let $\bal_n=(a_{n,0},a_{n,1},\dots,a_{n,n},0,0,\dots)$.
It suffices to show that $T_r(\bal_n)$ is totally positive for all $r\ge0$.
If $Q(A)$ is totally positive, then so is its leading principal submatrix $Q_n(A)$.
Hence, the block matrices on the right-hand-side of \eqref{eq-M} are all totally positive,
which implies that $M_{n,r}$ is totally positive, by Proposition \ref{prop-TP-prod}.
Therefore, by \eqref{eq-Toep},
the Toeplitz matrix $T_r(\bal_n)$ is totally positive,
since taking submatrix and transpose both preserve the total positivity of matrices.
This completes the proof.
\end{proof}

\section{Exponential Riordan arrays}
\hspace*{\parindent}
Riordan arrays play an important unifying role in enumerative combinatorics \cite{SGW+91,Spr94}.
Two main types of Riordan arrays are the (ordinary) Riordan arrays and the exponential Riordan arrays.
An \defn{(ordinary) Riordan array} $[r_{n,k}]_{n,k\ge0}$,
denoted by $\cR(d,h)$,
is an infinite matrix whose generating function of the $k$th column is $d(t)h^k(t)$,
where $d(0)\neq0$, $h(0)=0$ and $h'(0)\neq0$.
That is
\begin{equation}\label{eq-RA}
	r_{n,k}=[t^n]d(t)h^k(t).
\end{equation}
Let $g(t)=\sum_{n\ge0} g_n t^n/n!$
and $f(t)=\sum_{n\ge1} f_n t^n/n!$
be formal power series with $g_0\neq0$ and $f_1\neq0$.
The \defn{exponential Riordan array}
associated to the pair $(g,f)$ is the infinite lower triangular matrix $[R_{n,k}]_{n,k\ge0}$ defined by
\begin{equation}\label{eq-ERA}
	R_{n,k}=\frac{n!}{k!} [t^n]g(t)f^k(t).
\end{equation}
That is, the exponential generating function of the $k$th column is $g(t)f^k(t)/k!$.
Such an exponential Riordan array is denoted by $\cR[g,f]$.
For more on (exponential) Riordan arrays, we refer the reader to \cite{Bar16,SSB+22} and references therein.

By the fundamental theorem for exponential Riordan arrays,
the set of all exponential Riordan arrays forms a group,
called the \defn{exponential Riordan group},
under matrix multiplication
\begin{equation}\label{eq-FTERA}
	\cR[g_1,f_1]\cdot\cR[g_2,f_2]=\cR[g_1(g_2\circ f_1),f_2\circ f_1].
\end{equation}
The identity is the matrix $\cR[1,t]$,
and the inverse of $\cR[g,f]$ is $\cR[1/(g\circ\bar{f}),\bar{f}]$,
where $\bar{f}$ is the compositional inverse of $f$.
Well-known exponential Riordan arrays include
(see \cite{SSB+22} for instance)
\begin{itemize}
	\item Pascal triangle:
		  $\left[\binom{n}{k}\right]_{n,k\ge0}=\cR[e^t,t]$;
	\item (signless) Stirling triangle of the first kind:
		  $[c(n,k)]_{n,k\ge1}=\cR\left[\frac{1}{1-t},\ln{\frac{1}{1-t}}\right]$;
	\item Stirling triangle of the second kind:
		  $[S(n,k)]_{n,k\ge1}=\cR[e^t,e^t-1]$;
	\item (signless) Lah triangle:
		  $[L(n,k)]_{n,k\ge1}=\cR\left[\frac{1}{(1-t)^2},\frac{t}{1-t}\right]$.
\end{itemize}

It is not hard to show that the left production matrix of the Pascal triangle is
$$
J=
\left[
	\begin{array}{*{4}c}
		1 \\
		1 &  1 \\
		1 &  1 &  1 \\
		\vdots &&& \ddots
	\end{array}
\right]
$$
since $\binom{n}{k}=\sum_{i=k}^n \binom{i-1}{k-1}$.
Moreover, $J$ is totally positive since the sequence $1,1,1,\dots$ is PF as its generating function is $1/(1-t)$.
Then by Theorem \ref{thm-main},
the Pascal triangle is totally positive,
which has been proved by many authors \cite{Kar68,CLW15rio,CLW15rec,CW19};
and the polynomials $(1+x)^n$ are real-rooted,
which is obviously true.
On the other hand, the total positivity of the Pascal triangle,
as well as the Lah triangle,
can also be obtained by the following lemma.

\begin{lem}\label{lem-TP-RA}
	If both $g$ and $f$ are PF, then
	\begin{enumerate}[\rm(1)]
		\item the ordinary Riordan array $\cR(g,f)$ is totally positive;
		\item the exponential Riordan array $\cR[g,f]$ is totally positive.
	\end{enumerate}
\end{lem}

\begin{proof}
	(1) was proved by Chen and Wang in \cite{CW19}.
	(2) follows from the fact that the ordinary Riordan array $\cR(g,f)=[r_{n,k}]_{n,k\ge0}$ is totally positive if and only if the exponential Riordan array $\cR[g,f]=[R_{n,k}]_{n,k\ge0}$ is totally positive,
	which is clear since $R_{n,k}=\frac{n!}{k!} r_{n,k}$ by definition.
\end{proof}

It is easy to see that the Stirling triangles of both kinds and the Lah triangle all belong to the \defn{derivative subgroup} of the exponential Riordan group,
in which each matrix is of the form $\cR[f',f]$.
However, Lemma \ref{lem-TP-RA} does not work for their total positivity
whenever $f$ is not PF.
Now we present an alternative criterion for the total positivity of such matrices.

\begin{thm}\label{thm-ERA}
	Let $R=\cR[f',f]$ be an exponential Riordan array.
	If $f'$ is PF, then both $R$ and $\rev{R}$ are totally positive, and the row generating functions of $R$ are all real-rooted.
\end{thm}

\begin{proof}
	Note that \eqref{eq-FTERA} implies
	$$
	\cR[f',f]
	=\cR[f',t]\cdot\cR[1,f]
	=\cR[f',t]\cdot
	\left[
		\begin{array}{cc}
			1 \\
			  & \cR[f',f]
		\end{array}
	\right],
	$$
	or equivalently,
	$\cR[f',t]$ is the left production matrix of $\cR[f',f]$.
	By Lemma \ref{lem-TP-RA}, this left production matrix is totally positive
	since $f'$ and $t$ are both PF.
	Thus, applying Theorem \ref{thm-main} to $\cR[f',f]$,
	we complete the proof.
\end{proof}

\begin{cor}
	The Stirling triangles of both kinds, the Lah triangle, as well as their reversals are all totally positive.
	The Bell polynomials $B_n(x):=\sum_{k=1}^n S(n,k) x^k$ and
	the Lah polynomials $\sum_{k=1}^n L(n,k) x^k$ are all real-rooted.
\end{cor}

\begin{proof}
	All these triangles are of the form $\cR[f',f]$, and $f'$ are all PF.
\end{proof}

\begin{cor}\label{cor-ERA}
	Let $R=\cR[1,f]$ be an exponential Riordan array.
	If $f'$ is PF, then both $R$ and $\rev{R}$ are totally positive, and the row generating functions of $R$ are all real-rooted.
\end{cor}

\begin{proof}
	Note that
	\begin{equation}\label{eq-R1f}
	\cR[1,f]=
	\left[
		\begin{array}{cc}
			1 \\
			& \cR[f',f]
		\end{array}
	\right].
	\end{equation}
	Clearly, the total positivity of $\cR[1,f]$ and its reversal are equivalent respectively to that of $\cR[f',f]$ and its reversal;
	and the real-rootedness of the row generating functions of $\cR[1,f]$ is equivalent to that of $\cR[f',f]$.
	Thus, the proof is complete by Theorem \ref{thm-ERA}.
\end{proof}

The exponential Riordan arrays $\cR[1,f]$ are also considered as the iteration matrices defined by the exponential Bell polynomials \cite[p.133, 145]{Com74}. In the remaining part of this section, we discuss the total positivity and real-rootedness properties of the iteration matrices in terms of multiplier sequences.
Let $(\ga_k)_{k\ge0}$ be a sequence of real numbers.
Define an operator $\Ga:~\mathbb{R}[x]\to\mathbb{R}[x]$ by
$$
\Ga\left(\sum_{k=0}^n a_k x^k\right)=\sum_{k=0}^n \ga_k a_k x^k.
$$
Call $(\ga_k)_{k\ge0}$ a \defn{multiplier sequence} if the real-rootedness of $f(x)$ implies that of $\Ga(f(x))$.

\begin{prop}\label{prop-MS+1}
	If $(\ga_k)_{k\ge0}$ is a multiplier sequence, then so is $(\ga_{k+1})_{k\ge0}$.
\end{prop}

\begin{proof}
	Let $f(x)=\sum_{k=0}^n a_k x^k$ be real-rooted.
	Then $\sum_{k=0}^n \ga_{k+1} a_k x^k=\frac{1}{x}\Ga(xf(x))$ is also real-rooted.
\end{proof}

P\'olya and Schur obtained the transcendental charaterization of multiplier sequences (see also in \cite{Lev80}):
a sequence $(\ga_k)_{k\ge0}$ of nonnegative numbers is a multiplier sequence if and only if its exponential generating function 
\begin{equation}\label{eq-MS}
	\sum_{k\ge0} \frac{\ga_k}{k!} x^k=C x^n e^{\theta x} \prod_{i\ge1} (1+\rho_i x),
\end{equation}
where $C\ge0$, $n\in\mathbb{N}$, $\theta,\rho_i\ge0$ and $\sum_{i\ge1} \rho_i<\infty$.
Comparing \eqref{eq-MS} with \eqref{eq-PF},
the following proposition is immediate.

\begin{prop}\label{prop-MS/k!}
	If $(\ga_k)_{k\ge0}$ is a multiplier sequence of nonnegative numbers, then the sequence $(\ga_k/k!)_{k\ge0}$ is PF.
\end{prop}

The \defn{exponential partial Bell polynomials} are $B_{n,k}:=B_{n,k}(x_1,x_2,\dots,x_{n-k+1})$ in an infinite number of variables $x_1,x_2,\dots$,
defined by
$$
\sum_{n\ge k} B_{n,k}\frac{t^n}{n!}:=\frac{1}{k!}\left(\sum_{m\ge1} x_m \frac{t^m}{m!}\right)^k
$$
for $k=0,1,2,\dots$.
The matrix $[B_{n,k}]_{n,k\ge0}$ is called the \defn{iteration matrix} associated to $f(t)=\sum_{m\ge1} x_m t^m/m!$.
The exponential partial Bell polynomials generalize many well-known combinatorial numbers (see \cite[p. 135]{Com74} for instance):
\begin{itemize}
	\item Stirling numbers of the second kind:
		  $S(n,k)=B_{n,k}(1,1,1,\dots)$;
	\item idempotent numbers:
		  $\binom{n}{k}k^{n-k}=B_{n,k}(1,2,3,\dots)$;
	\item signless Lah numbers:
		  $\binom{n-1}{k-1}\frac{n!}{k!}=B_{n,k}(1!,2!,3!\dots)$;
	\item signless Stirling numbers of the first kind:
		  $c(n,k)=B_{n,k}(0!,1!,2!,\dots)$.
\end{itemize}

\begin{cor}\label{cor-MS}
	Let $B=[B_{n,k}(x_1,x_2,\dots)]_{n,k\ge0}$ be an iteration matrix.
	If $(x_1,x_2,\dots)$ is a multiplier sequence,
	then both $B$ and its reversal $\rev{B}$ are totally positive,
	and the polynomials $\sum_{k=0}^n B_{n,k} x^k$ are real-rooted for all $n\ge0$.
\end{cor}

\begin{proof}
	Note that in the sense of exponential Riordan arrays,
	the iteration matrix $B$ is precisely $\cR[1,f]$, where $f(t)=\sum_{m\ge1} x_m t^m/m!$.
	Thus,
	$$
	f'(t)=\sum_{m\ge1} x_m \frac{t^{m-1}}{(m-1)!}=\sum_{m\ge0} x_{m+1} \frac{t^m}{m!}.
	$$
	Suppose that $(x_{m+1})_{m\ge0}$ is a multiplier sequences.
	Then $(x_{m+1}/m!)_{m\ge0}$ is PF by Proposition \ref{prop-MS/k!}.
	That is, $f'(t)$ is a PF series.
	Hence, by Corollary \ref{cor-ERA}, the proof is complete.
\end{proof}

Clearly, the sequence $1,1,1,\dots$ is a multiplier sequence.
Then again, we can obtain the total positivity of the Stirling triangle of the second kind $[S(n,k)]$ and its reversal,
and the real-rootedness of the Bell polynomials.
Moreover, we have the following.

\begin{cor}
	The matrices $\left[\binom{n}{k}k^{n-k}\right]_{n,k\ge0}$
	and $\left[\binom{n}{k} (n-k)^k\right]_{n,k\ge0}$ are both totally positive.
	The polynomials $\sum_{k=0}^n \binom{n}{k}k^{n-k} x^k$ are real-rooted for all $n\ge0$.
\end{cor}

\begin{proof}
	Recall that $\binom{n}{k}k^{n-k}=B_{n,k}(1,2,3,\dots)$.
	So, by Corollary \ref{cor-MS}, it suffices to show that the sequence $(1,2,3,\dots)$ is a multiplier sequence.
	In fact, if $f(x)=\sum_{k=0}^n a_k x^k$ is real-rooted,
	then so is $xf'(x)=\sum_{k=0}^n ka_k x^k$,
	which implies that the sequence $(0,1,2,\dots)$ is a multiplier sequence.
	Therefore, by Proposition \ref{prop-MS+1},
	the sequence $(1,2,3,\dots)$ is also a multiplier sequence,
	as desired.
\end{proof}

For exponential Riordan arrays not in the derivative subgroup,
even though the above theorems can not be applied to,
we are still able to obtain the total positivity and real-rootedness properties by calculating their left production matrices.
Then our general result, Theorem \ref{thm-main}, will be helpful.
To illustrate this point, we take as examples a generalization of the (type A and B) Stirling numbers,
which, as well as their $q$-analogues, have been appeared in recent literature \cite{DZ24,SS24}.

The Whitney numbers (of the second kind) $W_m(n,k)$ of any fixed group of order $m$ satisfy the recurrence \cite{Dow73}
$$
W_m(n,k)=W_m(n-1,k-1)+(1+mk)W_m(n-1,k)
$$
with initial conditions $W_m(0,k)=\del_{0,k}$.
When $m=1$ and $m=2$, the Whitney numbers $W_m(n,k)$ reduce to the type A and type B Stirling numbers (of the second kind), respectively.
More generally, Cheon and Jung \cite{CJ12} gave combinatorial interpretations for the $r$-Whitney numbers $W_{m,r}(n,k)$ which satisfy the recurrence
\begin{equation}\label{eq-Whitney}
	W_{m,r}(n,k)=W_{m,r}(n-1,k-1)+(r+mk)W_{m,r}(n-1,k)
\end{equation}
with $W_{m,r}(0,k)=\del_{0,k}$.
They also showed that the $r$-Whitney matrix $[W_{m,r}(n,k)]_{n,k\ge0}$ is an exponential Riordan array $\cR[e^{rt},(e^{mt}-1)/m]$.
Moreover, we can obtain the left production matrix of the $r$-Whitney matrix.

\begin{prop}\label{prop-Whitney}
	The left production matrix of the $r$-Whitney matrix $[W_{m,r}(n,k)]_{n,k\ge0}$ is an ordinary Riordan array $\cR(1/(1-t),t/(1-mt))$ for any $r$.
\end{prop}

\begin{proof}
Suppose that
$$
Q=[Q_{n,k}]_{n,k\ge0}:=\cR\left( \frac{1}{1-t},\frac{t}{1-mt} \right).
$$
Then by the sequence characterization of Riordan arrays \cite{HS09} we derive that $Q$ satisfies the recurrence
\begin{equation}\label{eq-Q}
	Q_{n,0}=Q_{n-1,0}, \quad
	Q_{n,k}=Q_{n-1,k-1}+m Q_{n-1,k}
\end{equation}
with initial conditions $Q_{0,k}=\del_{0,k}$.
To prove that $Q$ is the left production matrix of $[W_{m,r}(n,k)]_{n,k\ge0}$,
it suffices to show that
\begin{equation}\label{eq-W-left}
	W_{m,r}(n,k)=\sum_{i=k}^n Q_{n,i} W_{m,r}(i-1,k-1)
\end{equation}
for $n,k\ge1$.
In fact, by induction on $n$, we have
\begin{align*}
	\text{RHS of \eqref{eq-W-left}}
	=&\sum_{i=k}^n (Q_{n-1,i-1}+mQ_{n-1,i})W_{m,r}(i-1,k-1) \\
	=&\sum_{i=k}^n Q_{n-1,i-1} [W_{m,r}(i-2,k-2)+(r+m(k-1))W_{m,r}(i-2,k-1)] \\
	&+m\sum_{i=k}^n Q_{n-1,i} W_{m,r}(i-1,k-1) \\
	=& W_{m,r}(n-1,k-1)+(r+m(k-1))W_{m,r}(n-1,k)+mW_{m,r}(n-1,k) \\
	=&W_{m,r}(n,k),
\end{align*}
where the first equation follows from \eqref{eq-Q},
the second and the last equations follow from \eqref{eq-Whitney},
and the third equation follows from the induction hypothesis.
This completes the proof.
\end{proof}

\begin{cor}
	For any $m,r\ge0$, the $r$-Whitney triangle $[W_{m,r}(n,k)]_{n,k\ge0}$ and its reversal are both totally positive, and the polynomials $\sum_{k=0}^n W_{m,r}(n,k)x^k$ are real-rooted for all $n\ge0$.
	In particular, the row generating functions of the (type A and type B) Stirling triangles are real-rooted.
\end{cor}

\begin{proof}
	Combining Proposition \ref{prop-Whitney} with Lemma \ref{lem-TP-RA} we obtain that the left production matrix $\cR(1/(1-t),t/(1-mt))$ is totally positive for all $m\ge0$.
	Then applying Theorem \ref{thm-main} to $[W_{m,r}(n,k)]_{n,k\ge0}$,
	the proof is complete.
\end{proof}

\section{$n$-recursive matrices}
\hspace*{\parindent}
Let $(a_n)_{n\ge1}$, $(b_n)_{n\ge1}$, $(c_n)_{n\ge2}$ be three sequences and define an infinite lower triangular matrix $T:=[t_{n,k}]_{n,k\ge0}$ by the recurrence
\begin{equation}\label{eq-n-rec}
	t_{0,0}=1, \quad
	t_{n,k}=a_n t_{n-1,k-1}+b_n t_{n-1,k}+c_n t_{n-2,k-1},
\end{equation}
where $t_{n,k}=0$ unless $n\ge k\ge0$.
For convenience, we call such matrices the \defn{$n$-recursive matrices}
since the coefficients of the recurrence only depend on $n$.
Such matrices arise often in combinatorics and the following are several basic examples of $n$-recursive matrices.

\begin{exa}\label{exa-n-rec}
	\begin{enumerate}[(i)]
		\item The Pascal triangle $\left[\binom{n}{k}\right]_{n,k\ge0}$ satisfies $\binom{n}{k}=\binom{n-1}{k-1}+\binom{n-1}{k}$.
		\item The (signless) Stirling triangle of the first kind $[c(n,k)]_{n,k\ge0}$ satisfies
			  $$
			  c(n,k)=c(n-1,k-1)+(n-1)c(n-1,k).
			  $$
		\item The (signless) type B Stirling triangle of the first kind $[c_B(n,k)]_{n,k\ge0}$ satisfies
			  $$
			  c_B(n,k)=c_B(n-1,k-1)+(2n-1)c_B(n-1,k).
			  $$
		\item The Delannoy triangle $[d(n,k)]_{n,k\ge0}$ satisfies $d(0,0)=1$ and
			  $$
			  d(n,k)=d(n-1,k-1)+d(n-1,k)+d(n-2,k-1).
			  $$
		\item The type A derangement triangle $[d_A(n,k)]_{n,k\ge0}$ satisfies $d_A(0,0)=1$ and
			  $$
			  d_A(n,k)=(n-1)d_A(n-1,k)+(n-1)d_A(n-2,k-1).
			  $$
		\item The type B derangement triangle $[d_B(n,k)]_{n,k\ge0}$ satisfies $d_B(0,0)=1$ and
			  $$
			  d_B(n,k)=d_B(n-1,k-1)+2(n-1)d_B(n-1,k)+2(n-1)d_B(n-2,k-1).
			  $$
	\end{enumerate}
\end{exa}

Brenti \cite[Theorem 4.3]{Bre95} showed the total positivity of $n$-recursive matrices and the real-rootedness of the row generating functions.
Moreover, we can obtain the left production matrices of both $n$-recursive matrices and their reversals.

\begin{prop}\label{prop-n-rec-leftprod}
	Let $T=[t_{n,k}]_{n,k\ge0}$ be an $n$-recursive matrix defined by \eqref{eq-n-rec}.
	Then the left production matrix of $T$ is
	$$
	Q(T)=\left[\frac{(\bbe_n)!}{(\bbe_k)!}\cdot{\rm I}[n\ge k]\right]_{n,k\ge0}\cdot
	\left[
		\begin{array}{cc}
			1 \\
			  & D(\bal,\bga)
		\end{array}
	\right],
	$$
	where $(\bbe_n)!:=b_1 b_2\cdots b_n$ with the convention that $(\bbe_0)!=1$, ${\rm I}[*]=1$ or 0 if the condition $*$ is true or false, and
	$$
	D(\bal,\bga)=
	\left[
		\begin{array}{cccc}
			a_1 \\
			c_2 & a_2 \\
			    & c_3 & a_3 \\
				&     & \ddots & \ddots
		\end{array}
	\right].
	$$
\end{prop}

\begin{proof}
	We present two proofs, one is algebraic and the other one is combinatorial.

	(Algebraic proof.)
	Let
	$$
	L(\bbe)=\left[\frac{(\bbe_n)!}{(\bbe_k)!} \cdot {\rm I}[n\ge k]\right]_{n,k\ge0}.
	$$
	Note that
	$$
	L(\bbe)^{-1}=D(1,-\bbe)=
	\left[
		\begin{array}{cccc}
			1 \\
			-b_1 & 1 \\
			     & -b_2 & 1 \\
				 &      & \ddots & \ddots
		\end{array}
	\right].
	$$
	Thus, to prove that
	$
	Q(T):=L(\bbe)
	\left[
		\begin{array}{cc}
			1 \\
			  & D(\bal,\bga)
		\end{array}
	\right]
	$
	is the left production matrix of $T$,
	it suffices to show that
	\begin{equation}\label{eq-T}
	T=L(\bbe)
	\left[
		\begin{array}{cc}
			1 \\
			  & D(\bal,\bga)
		\end{array}
	\right]
	\left[
		\begin{array}{cc}
			1 \\
			  & T
		\end{array}
	\right].
	\end{equation}
	Let $[u_{n,k}]_{n,k\ge0}:=L(\bbe)^{-1}\cdot T$ and $[v_{n,k}]_{n,k\ge0}:=D(\bal,\bga)\cdot T$, then
	$$
	u_{n,k}=t_{n,k}-b_n t_{n-1,k}=a_n t_{n-1,k-1}+c_n t_{n-2,k-1}
	$$
	and
	$$
	v_{n,k}=a_{n+1} t_{n,k}+c_{n+1} t_{n-1,k},
	$$
	which implies that $v_{n,k}=u_{n+1,k+1}$.
	Hence, \eqref{eq-T} holds as desired.

	(Combinatorial proof.)
	It is not hard to find a planar network for $T$ (see Figure \ref{fig-n-rec}(a) for instance).
	Then from the construction of planar networks in terms of the left production matrices,
	one can derive the planar network for the left production matrix $Q(T)$, as shown in Figure \ref{fig-n-rec}(b).
	Thus, by computing the path matrix, we have
	$
	Q(T)=
	L(\bbe)\cdot
	\left[
		\begin{array}{cc}
			1 \\
			  & D(\bal,\bga)
		\end{array}
	\right].
	$
\end{proof}

\begin{figure}[htbp]
\centering
\begin{tikzpicture}[scale=.6,>=latex]
	\foreach \y in {0,...,4}
		{
			\node[left] at (0,\y) {$u_{\y}$};
			\foreach \x in {0,...,4}
				{
					\filldraw (\x,\y) circle(.08);
					\draw (\x,\y)--(\x+1,\y);
				}
		}
	\foreach \y in {1,...,4}
		{
			\draw (4-\y,\y)-- node{\scriptsize $b_{\y}$} (5-\y,\y-1);
			\node at (4.5,\y) {\scriptsize $a_{\y}$};
		}
	\foreach \y in {2,3,4}
		\draw (4,\y)-- node{\scriptsize $c_{\y}$} (5,\y-1);
	\foreach \y in {0,...,4}
		{
			\foreach \x in {5,...,8}
				{
					\filldraw (\x,\y) circle(.08);
					\draw (\x,\y)--(\x+1,\y);
				}
		}
	\foreach \y in {1,...,3}
		{
			\draw (8-\y,\y+1)-- node{\scriptsize $b_{\y}$} (9-\y,\y);
			\node at (8.5,\y+1) {\scriptsize $a_{\y}$};
		}
	\foreach \y in {2,3}
		\draw (8,\y+1)-- node{\scriptsize $c_{\y}$} (9,\y);
	\foreach \y in {0,...,4}
		{
			\foreach \x in {9,...,11}
				{
					\filldraw (\x,\y) circle(.08);
					\draw (\x,\y)--(\x+1,\y);
				}
		}
	\foreach \y in {1,2}
		{
			\draw (11-\y,\y+2)-- node{\scriptsize $b_{\y}$} (12-\y,\y+1);
			\node at (11.5,\y+2) {\scriptsize $a_{\y}$};
		}
	\foreach \y in {2}
		\draw (11,\y+2)-- node{\scriptsize $c_{\y}$} (12,\y+1);
	\foreach \y in {0,...,4}
		{
			\foreach \x in {12,13}
				{
					\filldraw (\x,\y) circle(.08);
					\draw (\x,\y)--(\x+1,\y);
				}
			\filldraw (14,\y) circle(.08) node[right] {$v_{\y}$};
		}
	\foreach \y in {1}
		{
			\draw (13-\y,\y+3)-- node{\scriptsize $b_{\y}$} (14-\y,\y+2);
			\node at (13.5,\y+3) {\scriptsize $a_{\y}$};
		}
	\foreach \y in {0,...,4}
		{
			\foreach \x in {18,...,22}
				{
					\filldraw (\x,\y) circle(.08);
					\draw (\x,\y)--(\x+1,\y);
				}
			\filldraw (23,\y) circle(.08);
		}
	\foreach \y in {1,...,4}
		{
			\draw (22-\y,\y)-- node{\scriptsize $b_{\y}$} (23-\y,\y-1);
			\node at (22.5,\y) {\scriptsize $a_{\y}$};
		}
	\foreach \y in {2,3,4}
		\draw (22,\y)-- node{\scriptsize $c_{\y}$} (23,\y-1);
	\node at (7,-1) {(a)};
	\node at (20,-1) {(b)};
\end{tikzpicture}
\caption{(a) A planar network for the $n$-recursive matrix $T$. (b) The planar network for the left production matrix $Q(T)$.}
\label{fig-n-rec}
\end{figure}
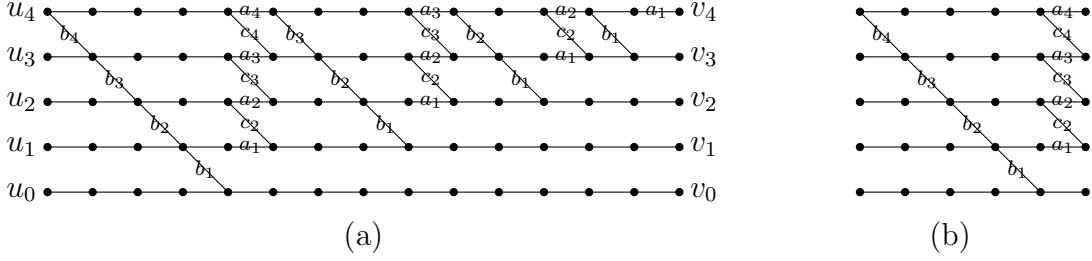

Now we consider the reversal of the $n$-recursive matrices,
and show combinatorially that the left production matrices of the reversal matrix and the original matrix have certain ``dual" relationship (by interchanging $a_n$ and $b_n$).

\begin{cor}
	Let $T=[t_{n,k}]_{n,k\ge0}$ be an $n$-recursive matrix defined by \eqref{eq-n-rec}.
	Then the left production matrix of $\rev{T}$ is
	$$
	Q(\rev{T})=\left[\frac{(\bal_n)!}{(\bal_k)!}\cdot{\rm I}[n\ge k]\right]_{n,k\ge0}\cdot
	\left[
		\begin{array}{cc}
			1 \\
			  & D(\bbe,\bga)
		\end{array}
	\right],
	$$
	where $(\bal_n)!:=a_1 a_2\cdots a_n$ with the convention that $(\bal_0)!=1$, ${\rm I}[*]=1$ or 0 if the condition $*$ is true or false, and
	$$
	D(\bbe,\bga)=
	\left[
		\begin{array}{cccc}
			b_1 \\
			c_2 & b_2 \\
			    & c_3 & b_3 \\
				&     & \ddots & \ddots
		\end{array}
	\right].
	$$
\end{cor}

\begin{proof}
	It is already known that Figure \ref{fig-n-rec}(a) is a planar network for $T$.
	Then from the construction of planar networks for reversal matrices in the proof of Theorem \ref{thm-main} (2),
	we derive a planar network for $\rev{T}$ as shown in Figure \ref{fig-n-rec-rev}(a),
	which is equivalent to the planar network in Figure \ref{fig-n-rec-rev}(b).
	Note that the planar networks in Figures \ref{fig-n-rec}(a) and \ref{fig-n-rec-rev}(b) are identical
	up to an interchange of $a_n$ and $b_n$.
	Hence, by Proposition \ref{prop-n-rec-leftprod}, we have the desired $Q(\rev{T})$.
\end{proof}

\begin{figure}[htbp]
\centering
\begin{tikzpicture}[scale=.6,>=latex]
	\foreach \y in {0,...,4}
		{
			\foreach \x in {0,...,4}
				{
					\filldraw (\x,\y) circle(.08);
					\draw (\x,\y)--(\x+1,\y);
				}
		}
	\foreach \y in {1,...,4}
		{
			\draw (4-\y,\y)-- node{\scriptsize $b_{\y}$} (5-\y,\y-1);
			\node at (4.5,\y) {\scriptsize $a_{\y}$};
		}
	\foreach \y in {2,3,4}
		\draw (4,\y)-- node{\scriptsize $c_{\y}$} (5,\y-1);
	\foreach \y in {0,...,4}
		{
			\foreach \x in {5,...,8}
				{
					\filldraw (\x,\y) circle(.08);
					\draw (\x,\y)--(\x+1,\y);
				}
		}
	\foreach \y in {1,...,3}
		{
			\draw (8-\y,\y+1)-- node{\scriptsize $b_{\y}$} (9-\y,\y);
			\node at (8.5,\y+1) {\scriptsize $a_{\y}$};
		}
	\foreach \y in {2,3}
		\draw (8,\y+1)-- node{\scriptsize $c_{\y}$} (9,\y);
	\foreach \y in {0,...,4}
		{
			\foreach \x in {9,...,11}
				{
					\filldraw (\x,\y) circle(.08);
					\draw (\x,\y)--(\x+1,\y);
				}
		}
	\foreach \y in {1,2}
		{
			\draw (11-\y,\y+2)-- node{\scriptsize $b_{\y}$} (12-\y,\y+1);
			\node at (11.5,\y+2) {\scriptsize $a_{\y}$};
		}
	\foreach \y in {2}
		\draw (11,\y+2)-- node{\scriptsize $c_{\y}$} (12,\y+1);
	\foreach \y in {0,...,4}
		{
			\foreach \x in {12,13}
				{
					\filldraw (\x,\y) circle(.08);
					\draw (\x,\y)--(\x+1,\y);
				}
			\filldraw (14,\y) circle(.08);
		}
	\foreach \y in {1}
		{
			\draw (13-\y,\y+3)-- node{\scriptsize $b_{\y}$} (14-\y,\y+2);
			\node at (13.5,\y+3) {\scriptsize $a_{\y}$};
		}
	\node[above] at (14,4) {$u_0$};
	\node[above] at (12,4) {$u_1$};
	\node[above] at (9,4) {$u_2$};
	\node[above] at (5,4) {$u_3$};
	\node[above] at (0,4) {$u_4$};
	\foreach \y in {0,...,4}
		\node[right] at (14,4-\y) {$v_{\y}$};
	\node at (7,-1) {(a)};
\end{tikzpicture}

\begin{tikzpicture}[scale=.6,>=latex]
	\foreach \y in {0,...,4}
		{
			\node[left] at (0,\y) {$u_{\y}$};
			\foreach \x in {0,...,4}
				{
					\filldraw (\x,\y) circle(.08);
					\draw (\x,\y)--(\x+1,\y);
				}
		}
	\foreach \y in {1,...,4}
		{
			\draw (4-\y,\y)-- node{\scriptsize $a_{\y}$} (5-\y,\y-1);
			\node at (4.5,\y) {\scriptsize $b_{\y}$};
		}
	\foreach \y in {2,3,4}
		\draw (4,\y)-- node{\scriptsize $c_{\y}$} (5,\y-1);
	\foreach \y in {0,...,4}
		{
			\foreach \x in {5,...,8}
				{
					\filldraw (\x,\y) circle(.08);
					\draw (\x,\y)--(\x+1,\y);
				}
		}
	\foreach \y in {1,...,3}
		{
			\draw (8-\y,\y+1)-- node{\scriptsize $a_{\y}$} (9-\y,\y);
			\node at (8.5,\y+1) {\scriptsize $b_{\y}$};
		}
	\foreach \y in {2,3}
		\draw (8,\y+1)-- node{\scriptsize $c_{\y}$} (9,\y);
	\foreach \y in {0,...,4}
		{
			\foreach \x in {9,...,11}
				{
					\filldraw (\x,\y) circle(.08);
					\draw (\x,\y)--(\x+1,\y);
				}
		}
	\foreach \y in {1,2}
		{
			\draw (11-\y,\y+2)-- node{\scriptsize $a_{\y}$} (12-\y,\y+1);
			\node at (11.5,\y+2) {\scriptsize $b_{\y}$};
		}
	\foreach \y in {2}
		\draw (11,\y+2)-- node{\scriptsize $c_{\y}$} (12,\y+1);
	\foreach \y in {0,...,4}
		{
			\foreach \x in {12,13}
				{
					\filldraw (\x,\y) circle(.08);
					\draw (\x,\y)--(\x+1,\y);
				}
			\filldraw (14,\y) circle(.08) node[right] {$v_{\y}$};
		}
	\foreach \y in {1}
		{
			\draw (13-\y,\y+3)-- node{\scriptsize $a_{\y}$} (14-\y,\y+2);
			\node at (13.5,\y+3) {\scriptsize $b_{\y}$};
		}
	\node at (7,-1) {(b)};
\end{tikzpicture}
\caption{(a) A planar network for $\rev{T}$. (b) An equivalent planar network to (a).}
\label{fig-n-rec-rev}
\end{figure}
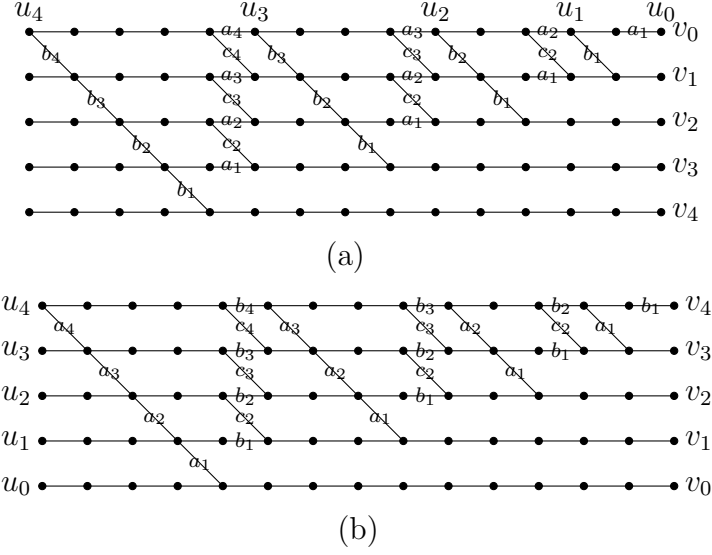

\begin{cor}\label{cor-n-rec}
	Let $T=[t_{n,k}]_{n,k\ge0}$ be an $n$-recursive matrix defined by \eqref{eq-n-rec}.
	If $\bal=(a_n)_{n\ge1}, \bbe=(b_n)_{n\ge1}, \bga=(c_n)_{n\ge2}$ are all nonnegative sequences,
	then both $T$ and $\rev{T}$ are totally positive,
	and the polynomials $\sum_{k=0}^n t_{n,k} x^k$ are real-rooted for all $n\ge0$.
\end{cor}

\begin{proof}
	By Theorem \ref{thm-main}, it suffices to show that the left production matrix $Q(T)$ is totally positive.
	Recall that
	$$
	Q(T)=L(\bbe)\cdot
	\left[
		\begin{array}{cc}
			1 \\
			  & D(\bal,\bga)
		\end{array}
	\right],
	$$
	where $L(\bbe)$ is the signless inverse of $D(1,\bbe)$.
	It is known \cite[Proposition 1.6]{Pin10} that a matrix is totally positive if and only if its signless inverse is totally positive.
	Thus, both $L(\bbe)$ and
	$
	\left[
		\begin{array}{cc}
			1 \\
			  & D(\bal,\bga)
		\end{array}
	\right]
	$
	are totally positive,
	since the bidiagonal matrices $D(1,\bbe)$ and $D(\bal,\bga)$ are totally positive provided that the sequences $\bal,\bbe,\bga$ are nonnegative.
	Therefore, by Proposition \ref{prop-TP-prod}, the product $Q(T)$ is also totally positive, as desired.
\end{proof}

Applying Corollary \ref{cor-n-rec} to the triangles in Example \ref{exa-n-rec},
we have the following.

\begin{cor}
	The Pascal triangle, the (type A and type B) Stirling triangles of the first kind, the Delannoy triangle, and the (typeA and type B) derangement triangles  are all totally positive,
	and so are their reversals.
	The row generating functions of the above triangles are all real-rooted.
\end{cor}

\section*{Acknowledgement}
This work was supported in part by the National Natural Science Foundation of China (No. 12271078).

\end{document}